\documentclass[final]{siamltex}
\usepackage{animate}
\usepackage{amsmath}
\usepackage{amssymb}
\usepackage{graphics}
\usepackage{subfigure}
\usepackage{graphicx}
\usepackage{textcomp}
\usepackage{mathrsfs}
\usepackage{epstopdf}
\usepackage{array}
\usepackage{cite}
\usepackage[maxfloats=99]{morefloats}
\usepackage{color}
\usepackage{url}

\renewcommand{\theequation}{\arabic{section}.\arabic{equation}}

\def\cal{\mathcal}

\def\to{{\longrightarrow}}

\def\E{{\mathcal E}}

\def\T{{\mathcal T}}

\def\N{{\mathcal N}}

\newcommand\ddelta\bigtriangledown
\newcommand\ld\lambda
\newcommand\Ld\Lambda

\def \bR {\Bbb R}

\newtheorem{remark}{Remark}[section]

\title{A  $C^0$ linear finite element method  for  sixth order elliptic equations}
\author{Hailong Guo\thanks{ Department of Mathematics, University of California Santa Barbara, CA, 93106 (hlguo@math.ucsb.edu). }
    \and Zhimin Zhang \thanks{Beijing Computational Science Research Center, Beijing 100094 and Department of Mathematics, Wayne State University, Detroit, MI 48202 (zzhang@math.wayne.edu). The research of this author was supported in part by the National Natural Science Foundation of China under grants 11471031, 91430216, U1530401, and the U.S. National Science Foundation through grant DMS-1419040.}
 \and Qingsong Zou\thanks{Corresponding author. School of Data and Computers and Guangdong Province Key Laboratory of Computational Science, Sun Yat-sen University, Guangzhou 510275 (mcszqs@mail.sysu.edu.cn). This author was partially supported  by
          the National Natural Science Foundation of China through grants 11571384 and 11428103,  by  Guangdong Provincial Natural
Science Foundation of China through grant 2014A030313179, and by the Fundamental Research Funds for the Central Universities through the grant 16lgjc80.}}

\begin{document}

\maketitle

%
%
\medskip

\begin{abstract} In this paper, we develop a  straightforward $C^0$ linear finite element method for sixth-order elliptic equations.
The basic idea  is to use gradient recovery techniques to generate  higher-order numerical derivatives from a $C^0$ linear finite element function.
 Both theoretical analysis and numerical experiments show that the proposed method has  the optimal convergence rate under the energy norm. The method avoids complicated construction of conforming $C^2$ finite element basis or  nonconforming penalty terms and has a low
 computational cost.


\vskip .3cm

\end{abstract}

\section{Introduction}
\setcounter{equation}{0}
Partial differential equations (PDEs) with order higher than 2 have been
 widely used to describe different physical laws in material sciences \cite{Cahn1961,CaginalpFife1986,HilliardCahn1958,HilliardCahn1959,WiseLowengrub2004},
elastic mechanics\cite{VentselKrauthammer2001}, quantum mechanics\cite{FushchychSymenoh1997}, plasma physics\cite{BiskampSchwarzDrake1995,BiskampSchwarzZeilerCelani1999,DrakeBiskampZeiler1991}, differential geometry \cite{ChangChen2001, Ugail2011},
 and other areas of science and engineering. Comparing with  the second-order PDEs, higher order PDEs are much less studied, including some fundamental
theoretical issues such as existence, uniqueness, and regularity of solutions.

Numerical simulation becomes  an important  tool  to study high order PDEs, and yet
the design of efficient and reliable numerical methods  is  very challenging.
As usual, the finite element method (FEM) plays a critical role in the numerical simulation. Both conforming and nonconforming methods have been applied to solve high order PDEs in the literature.
Usually, a conforming method  requires higher regularity of the approximating
functions (e.g. $C^1$ functions for a fourth-order PDE and $C^2$ for a sixth-order PDE), while a nonconforming method  avoids the construction of higher regularity finite elements by adding some specially designed penalty terms to the scheme.
The complicated construction of high regularity finite elements (for conforming methods) or penalty terms (for nonconforming methods)
makes these two FEMs hard to be implemented and  significantly increases computational cost. Moreover, the analysis of the aforementioned FEMs is often very complicated.

In this  work, we present a systematic and simple numerical approach to treat high-order PDEs and shed some light on
 theoretical analysis for this new method.
To be more precise, we will develop  a gradient recovery technique based $C^0$ linear finite element method for the
following sixth-order equation
\begin{eqnarray}\label{trihar}
-\triangle^3u &=&f\ \ \ \ {\rm in}\ \Omega\\
u=\partial_{\bf n} u=\partial^2_{\bf nn} u&=&0\ \ \ \ {\rm on}\
\partial\Omega,\label{hbc}
\end{eqnarray}
where $\Omega\subset \bR^2$ is an open bounded domain, $f\in L^2(\Omega)$, ${\bf n}$ is the outward unit normal of the boundary $\partial\Omega$.
The sixth order derivative  is defined as
\[
\Delta^3 u=\Delta(\Delta(\Delta u))=\sum_{i,j,k=1}^2\frac{\partial^6 u}{\partial x_i^2\partial x_j^2\partial x_k^2}
\]
and the directional derivatives are $\partial_{\bf n} u=\nabla u\cdot {\bf n}, \partial^2_{\bf nn} u={\bf n}^TD^2u\cdot {\bf n}$.
The (weak) solution of \eqref{trihar}-\eqref{hbc} is a function $u\in H^3_0(\Omega)$ satisfying
\begin{equation}\label{variational}
a(u,v)=(f,v), \forall v\in H_0^3(\Omega),
\end{equation}
where the third order derivative tensor  is given by $D^3v=\frac{\partial^3 v}{\partial x_i\partial x_j\partial x_k}$ and the bilinear form is 
\[
a(v,w)=\int_\Omega D^3v:D^3w, \quad \forall v,w\in H^3(\Omega).
\]
Here the Frobenius product ``:" for two tensors $B_1=(b_{ijk}^1),B_2=(b_{ijk}^2)$ is defined
as
\[
B_1:B_2=\sum_{i,j,k=1}^2 b^1_{ijk}b^2_{ijk}.
\]

Note that the sixth-order elliptic boundary value problem \eqref{trihar}  arises from
many mathematical models including  differential geometry(\cite{ChangChen2001, Ugail2011}),
the thin-film equations \cite{BarrettLangdon2004}, and the phase field crystal model \cite {BackofenRatzVoigt2007,ChengWarren2008,WiseWangLowengrub2009}. For simplicity,
we choose the homogeneous Dirichlet boundary conditions. The basic principle can be applied to  other
boundary conditions as well.

Let us illustrate the basic idea of the construction of our novel method. Note that the bilinear  form  $a(\cdot,\cdot)$
involves the third derivative of the discrete solution, which is impossible to obtain from a direct calculation of $C^0$ linear element whose gradient  is piecewise constant (w.r.t  the underlying mesh) and discontinuous across each element. To overcome this difficulty,  we use the gradient recovery operator $G_h$ to ``lift" discontinuous piecewise constant $D v_h$ to continuous piecewise linear function $G_hv_h$, see \cite{ZZ1992, Ainsworth-Oden,BabuskaStrouboulis,BankWeiser1985,BankXu2003a,BankXu2003b,comsol08, comsol14,Zhang} for the details of different recovery operators.
In other words, we use the special {\it difference operator} $DG_h^2$ to discretize the third order differential operator $D^3$.
Our algorithm  is then designed by applying  this special {\it difference operator} to the standard Ritz-Galerkin method.

 From the above construction, our method has some obvious advantages. First, the  fact that the recovery operator  $G_h$  can be defined on a general unstructured grid implies that the method  is valid for  problems on  arbitrary domains and meshes. Second, our method only  has function value unknowns on nodal points instead of both function value and derivative unknowns, its computational complexity is much lower than existing conforming and non-conforming methods in the literature.

Naturally, one may question on the consistency, stability, and convergence of  the proposed method,
which require some more in-depth mathematical analysis.
Let us begin with a discussion of consistency. As indicated in \cite{Zhang} (resp. \cite{guozhangzhao2014}), for reasonably regular meshes, $G_hu_I$ (resp. $G_h^2u_I$) is a second-order finite difference scheme of the gradient $Du$ (resp. $D^2u$), if $u$ is sufficiently smooth. Here
$u_I$ is the interpolation of $u$ in linear finite element space. As a consequence, $DG_h^2u_I$  is a first-order approximation of $D^3u$, provided $u$ is sufficiently smooth.
However,  for a discrete function $v_h$ in the finite element space which is not smooth across the element edges,
the error $\|Dv_h-G_h v_h\|_0$ is not a small quantity of the high order, sometimes it may not converge to zero at all. Fortunately,  an error estimate
in \cite{GuoZhangZou2015, ChenGuoZhangZou2017} set up a
consistency  in a weak sense, see \eqref{a2} and \eqref{a3} for the details.
This weak consistency property of the gradient recovery operator  will play an important role in our error analysis.

Next we discuss stability, which in our case can be reduced to
 verification of the (uniform) coercivity of the bilinear form $(DG_h^2\cdot, DG_h^2 \cdot)$ in the following sense
\begin{equation}\label{stability1}
\|v_h\|_0 \lesssim \|DG_h^2 v_h\|_0,
\end{equation}
for all $v_h$ in  the finite element space with suitable boundary conditions.
Again, since the {\it discrete Poincar\'e inequality} \eqref{a1}
has been established in \cite{GuoZhangZou2015}, the stability \eqref{stability1} is a direct consequence.
Note that  \eqref{stability1} implies that no additional penalty term is needed in order to guarantee the stability,
and this fact makes our  method very simple.

The  convergence properties of our method depends heavily on the aforementioned consistency, stability, and the nice approximation
properties of the recovery operator $G_h$.
As usual, the analysis of the error between the exact and approximate solutions can be decomposed into the analysis of the
approximation error and consistency error. Combining the weak consistency error estimates \eqref{a2},\eqref{a3} and approximation
error estimates \eqref{gr1}-\eqref{gr3} leads to the optimal convergence rate ($=1$) under the energy norm ($H^3$ norm).
This convergence rate is observed numerically. Furthermore, we also notice a second-order convergence rate under both $H^1$ and $L^2$ norms.
However, we are only able to prove a sub-optimal convergence rate $\frac32$ under both the $H^1$ and $L^2$ norms at this moment.
We would like to emphasize that our analysis here is straightforward and simpler than  the analysis of traditional conforming and nonconforming methods applied to sixth-order PDEs.

The rest of the paper is organized as follows. We first present  our algorithm in Section 2.
Several numerical examples are provided in Section 3
to illustrate the efficiency and convergence rates of our algorithm.
In Section 4, a rigorous mathematical analysis of our algorithm is given.
Finally, some concluding remarks are presented in the final section.

\section{A recovery based $C^0$ linear FEM}
In this section, we  discretize the variational equation \eqref{variational}
in the standard $C^0$ linear finite element space.

Let  $\T_h$ be a {\it triangulation} of  $\Omega$ with mesh-size $h$.
We denote by $\N_h$ and $\E_h$ the set of vertices and edges of $\T_h$, respectively.
Let $V_h$ be the standard ${\mathcal P}_1$ finite element space
corresponding to  $\T_h$. It is well-known that $V_h={\rm Span}\{\phi_p:p\in\N_h\}$ with $\phi_p$ a linear nodal
basis corresponding to each vertex $p\in\N_h$.
Let $G_h: V_h\to { V_h\times V_h}$ be a gradient recovery operator defined as below (\cite{ZZ1992,NagaZhang2005}).
For each vertex $p\in \N_h$, we define a recovered derivative $(G_hv_h)(p)$ and let the whole recovered gradient function  be
\[
G_h v_h=\sum_{p\in\N_h} (G_hv_h)(p)\phi_p.
\]

For all $v_h\in V_h$, we have $G_h v_h=(G_h^{x_1} v_h,G_h^{x_2} v_h)\in { V_h\times V_h}$. The corresponding recovered  Hessian  matrix
is defined as follows \cite{guozhangzhao2014}:
\[
G_h^2v_h=\left(\begin{array}{ll}
 G_h^{x_1}G_h^{x_1} v_h & G_h^{x_1} G_h^{x_2} v_h\\
G_h^{x_2} G_h^{x_1} v_h & G_h^{x_2} G_h^{x_2} v_h
\end{array}\right).
\]
The  derivative of $G_h^2 v_h$ is a  tensor with its component
\[
(DG_h^2v_h)_{ijk}=
\partial_{x_i} G_h^{x_j}G_h^{x_k} v_h , i,j,k=1,2.
\]
For all $v_h,w_h\in V_h$, we define a {\it discrete} bilinear form
\[
a_h(v_h,w_h)=\int_{\Omega} D(G_h^2v_h):D(G_h^2 w_h),
\]

The gradient recovery linear element scheme for solving \eqref{trihar} reads as :
Find
$u_h\in V_h^{0}$ such that
\begin{eqnarray}\label{grs4ebc}
a_h(u_h,v_h)&=&(f,v_h),\quad v_h\in
V_h^{0},
\end{eqnarray}
where the homogenous finite element space
\[
V_h^{0}=\{v_h\in V_h| v_h=G_hv_h\cdot{\bf n}=G_hv_h\cdot{\bf t}={\bf n}^T G_h^2 v_h {\bf n}=0 \ {\rm on}\ \partial\Omega\}.
\]
Note that here we use an additional condition $G_hv_h\cdot{\bf t}=0$
since the exact solution satisfies $\frac{\partial u}{\partial {\bf t}}=0$ on $\partial\Omega$,
where ${\bf t}$ is the unit tangential vector on $\partial \Omega$.

\begin{remark}
 For sixth order partial differential equation \eqref{trihar}-\eqref{hbc}, all three boundary conditions are essential boundary conditions and we should
 incorporate such types boundary conditions into the discretized linear system instead of the weak form.  For  partial differential equation \eqref{trihar}
 with nonhomogeneous boundary conditions
 \begin{equation}
 u|_{\partial \Omega}=g_D,\quad  \partial_{\bf n} u|_{\partial \Omega}=g_N, \quad  \partial^2_{\bf nn} u|_{\partial \Omega}= g_R.
\label{equ:nonhomo}
\end{equation}
The variation form is  to find $u_h \in V_h  $ with $u_h|_{\partial \Omega}=g_D,\partial_{\bf n} u_h|_{\partial \Omega}=g_N, \partial^2_{\bf nn} u_h|_{\partial \Omega}= g_R$
such that
\begin{eqnarray*}
a_h(u_h,v_h)&=&(f,v_h),\quad v_h\in
V_h^{0},
\end{eqnarray*}
For numerical implement, there is no difference between homogeneous and nonhomogeneous boundary conditions.  In the article, we suppose homogeneous boundary
conditions only for simplifying numerical analysis.
\end{remark}

\begin{remark}The scheme \eqref{grs4ebc} depends on the definition of $(G_hv_h)(p)$ at each vertex $p\in\N_h$.  In the following, three popular
definitions of $(G_hv_h)(p)$  are listed (c.f., \cite{ZZ1992,NagaZhang2005}).

(a) Weighted averaging(WA). For each $p\in \N_h$, let the element patch $\omega_p=\cup\{\tau:p\in\bar{\tau}\}$
and define
\begin{equation}\label{weight0}
(G_hv_h)(p)=\frac{1}{|\omega_p|}\int_{\omega_p} \nabla v_h (x_1,x_2) dx_1dx_2.
\end{equation}
\noindent
(b) Local $L^2$-projection. We seek two polynomials $P_l \in {\mathcal P}^1(\omega_p), (l=1,2)$, such that
\begin{equation}\label{lol2proj}
  \int_{\omega_p}[P_l(x_1,x_2)- \partial_{x_l}v(x_1,x_2)]Q(x_1,x_2)dx_1dx_2=0,\quad \forall Q\in \mathcal P^1(\omega_p), l=1,2
\end{equation}
  and  we  define
  \[
  (G_hv_h)(p) = (P_1(p),P_2(p)).
  \]
Sometimes, the exact integral in \eqref{lol2proj} is replaced by its discrete counterpart so that the
two polynomials $P_l,l=1,2$ satisfying  the least square fitting equation (SPR)
\begin{equation}\label{lsf}
 \sum_{i=1}^m[P_l(x^i_1,x^i_2)- \partial_{x^i_l}v(x^i_1,x^i_2)]Q(x^i_1,x^i_2)=0,\quad \forall Q\in \mathcal P^1(\omega_P), l=1,2,
\end{equation}
where $(x_1^i,x_2^i), i=1,\ldots,m$ are $m$ given points in $\omega_p$.

\noindent
(c) The polynomial preserving recovery (PPR). We seek a quadratic function $P\in \mathcal P^2(\omega_p)$, such that
 \begin{equation}\label{ppr}
  \sum_{i=1}^m[P(x_1^i,x_2^i)- v(x_1^i,x_2^i)]Q(x_1^i,x_2^i)=0,\quad \forall Q\in \mathcal P^2(\omega_P).
 \end{equation}
  Then we can define $(G_hv_h)(p) = (\partial_{x_1}P (p),\partial_{x_2} P(p))$.

It is known that the above three definitions are equivalent on a  mesh of uniform triangular pattern \cite{Zhang}.

\end{remark}

\begin{remark}Essentially, the operator $G_h$ can be regarded as a difference operator defined on unstructured grids. This operator
 lifts discontinuous gradient generated from a $C^0$-FEM to a continuous one, and thereby makes the further calculation of high order derivatives possible.
\end{remark}
\begin{remark}
The scheme \eqref{grs4ebc} is very simple and straightforward. It avoids the complicated construction of conforming $C^2$ finite element basis (c.f., \cite{HuZhang2015}) or the complicated construction of nonconforming penalty terms (\cite{GudiNeilan2011}).
\end{remark}

For $\mathcal{A}\subset \Omega$, let $V_h(\mathcal{A})$ denote the restrictions
of functions in $V_h$ to $\mathcal{A}$ and let $V_h^{\text{comp}}(\mathcal{A})$ denote the set of those functions
in $V_h(\mathcal{A})$ with compact support in the interior of $\mathcal{A}$
\cite{wahlbin1995}. Let
$ \Omega_0\subset\subset \Omega_1 \subset\subset \Omega_2 \subset\subset\Omega$
be  separated by $d\ge c_oh$ and  $\ell$ be a direction, i.e., a unit vector
in $\mathbb{R}^2$.
Let $\tau$ be a parameter, which will typically be a multiply of $h$.
Let $T^{\ell}_{\tau}$ denote translation by $\tau$ in the direction $\ell$,
i.e.,
\begin{equation}
  T^{\ell}_{\tau}v(x) = v(x+\tau\ell),
  \label{equ:translation}
\end{equation}
and for an integer $\nu$
\begin{equation}
  T^{\ell}_{\nu\tau}v(x) = v(x+\nu\tau\ell).
  \label{trans}
\end{equation}
Following the definition of \cite{wahlbin1995}, the finite element space
$V_h$ is called translation invariant by $\tau$ in the direction $\ell$ if
\begin{equation}
  T^{\ell}_{\nu \tau}v\in V^{\text{comp}}_h(\Omega), \quad
  \forall v \in V^{\text{comp}}_h(\Omega_1),
\end{equation}
for some integer $\nu$ with $|\nu| < M$.  Equivalently,
$\mathcal{T}_h$ is called a translation invariant mesh.
As illustrated in \cite{guozhangzhao2014},  uniform meshes of regular pattern, chevron pattern, cirsscross patter, and unionjack pattern are all translation invariant.

\section{Analysis}

The  section is dedicated to a mathematical proof for the convergence properties.

To this end, we need some properties of $G_h$.    For the polynomial preserving recovery operator  $G_h$, there are the following boundedness
property  (see (2.11) in \cite{NagaZhang2005})
\begin{equation}\label{gr0}
\|G_h v_h\|_0\lesssim |v_h|_1 ,\quad v_h\in V_h
\end{equation}
and
the superconvergence approximation properties
\begin{equation}\label{gr1}
\| \nabla u-G_h u_I\|_0\lesssim h^2|u|_{3,\infty}{,} \quad u\in W^{3,\infty}(\Omega).
\end{equation}
Here $u_I$ is the linear interpolation of $u$ in $V_h$.  In addition,   we will utilize the following ultraconvergence approximation properties of Hessian recovery
operator (see Theorem 3.5 in  \cite{guozhangzhao2014})
\begin{eqnarray}
\| D^3 u-DG^2_h u_I\|_0&\lesssim& h|u|_{4,\infty}{,} \quad u\in W^{4,\infty}(\Omega),\label{gr2}\\
\| D^2 u-G^2_h u_I\|_0&\lesssim& h^2|u|_{4,\infty}{,} \quad u\in W^{4,\infty}(\Omega),\label{gr3}
\end{eqnarray}
provided the mesh $\mathcal{T}_h$ is translation invariant.
\begin{remark}
 We would like to comment that the requirement $\mathcal{T}_h$ translation invariant for proving approximation properties \eqref{gr2} and \eqref{gr3} only for
 theoretical purpose.  In practical, our method can be applied to and shows optimal convergence on  arbitrary unstructured mesh.
\end{remark}


To analyze the convergence of the scheme \eqref{grs4ebc}, we suppose that the $\T_h$ is sufficient regular such that  there holds the following discrete Poincar\'{e} inequality (cf.,\cite{GuoZhangZou2015})
\begin{equation}\label{a1}
\ \ \  \|v_h\|_i\lesssim \|G_h v_h\|_i, \forall v_h\in V_h^0, i=0,1
\end{equation}
 and discrete {\it weak} approximation properties
\begin{eqnarray}\label{a2}
\left|\int_{\Omega} \nabla v\cdot (G_h v_h-\nabla v_h)\right|&\lesssim& h\|v\|_2 |G_hv_h|_1,\forall v\in H^2,\\
\label{a3}
\left|\int_{\Omega} \nabla v\cdot (G_h v_h-\nabla v_h)\right|&\lesssim& h^2\|v\|_3 |G_hv_h|_1,\forall v\in H^3.
\end{eqnarray}
Note that both \eqref{a1} and \eqref{a2} have been discussed in the analysis of a recovery operator based
linear finite element method for the biharmonic equation by Guo et al. in \cite{GuoZhangZou2015}.
In their paper, a counter-example shows that the strong error
 $\|\nabla v_h-G_hv_h\|_0$ is not necessary of ${\cal O}(h)$ for all $v_h\in V_h$ (which means the weak estimate \eqref{a2} might be
 the best estimate of the difference $\nabla v_h-G_hv_h$).

By \eqref{a1}, for all $v_h\in V_h^0$,
we have
\[
\|v_h\|_0\lesssim \|G_hv_h\|_0\lesssim \|G_h^2 v_h\|_0\lesssim \|DG_h^2 v_h\|_0.
\]
In other words, the semi-norm $\|DG_h^2\cdot\|_0$ is a norm.
 Then by the Lax-Milgram theorem,  the scheme \eqref{grs4ebc}  has a unique solution .
Moreover, by \eqref{grs4ebc},
\[
\|DG_h^2u_h\|_0^2=a_h(u_h,u_h)=(f,u_h)\lesssim \|f\|_0\|v_h\|_0.
\]
Then
\begin{equation}\label{stability}
\|DG_h^2 u_h\|_0\lesssim \|f\|_0
\end{equation}
which implies the stability of our scheme.

\subsection{$H^3$ error estimate}
\begin{theorem}
Let $u_h$ be the solution of \eqref{grs4ebc} and $u\in H^6$  the solution of \eqref{variational}.
 If the mesh $\T_h$ is  translation invariant, $G_h$ is properly defined  such that \eqref{gr0}-\eqref{a3} hold, then
\begin{equation}\label{h3}
\|D G_h^2(u_h-u_I)\|_0\lesssim h\|u\|_6,
\end{equation}
where $u_I$ is the linear interpolation of $u$ in $V_h$.
Consequently,
\begin{equation}\label{h30}
\|D^3 u-D G_h^2 u_h\|_0\lesssim h\|u\|_6.
\end{equation}
\label{thm:h3error}
\end{theorem}
\begin{proof}Since a weak solution of \eqref{variational} which has regularity $u\in H^6$
is also the strong solution  satisfying \eqref{trihar} and $u_h$ is a discrete solution
satisfying \eqref{grs4ebc}, we have
\begin{eqnarray*}
a_h(u_h, v_h)=(f,v_h)=(-\Delta^3u, v_h), \forall v_h\in V_h^0.
\end{eqnarray*}
Using the fact that $v_h\in V_h^0$,
we have
\[
(-\Delta^3u, v_h)=(-(\nabla\cdot\nabla)(\Delta^2u),  v_h)=(\nabla(\Delta^2u), \nabla v_h).
\]
Therefore,
\[
a_h(u_h, v_h)=I_1+(\nabla(\Delta^2u), G_h v_h),
\]
with
\begin{equation}\label{i1}
I_1=(\nabla(\Delta^2u), \nabla v_h-G_hv_h).
\end{equation}

Now we deal with the term  $(\nabla(\Delta^2u), G_h v_h)$.
Since $G_hv_h\cdot{\bf n}=G_hv_h\cdot{\bf t}=0$  on the boundary $\partial\Omega$, we have that on $\partial\Omega$,
\begin{eqnarray*}
\frac{\partial \Delta u}{\partial {\bf n}}\cdot G_h v_h=\frac{\partial^2 \Delta u}{\partial {\bf n}\partial {\bf t}}(G_h v_h\cdot {\bf t})+\frac{\partial^2 \Delta u}{\partial {\bf n}^2}(G_h v_h\cdot {\bf n})=0.
\end{eqnarray*}
Then
\begin{eqnarray*}
(\nabla(\Delta^2u), G_h v_h)&=&((\nabla\cdot\nabla)(\Delta (\nabla u)), G_h v_h)\\
&=&
- (D^2(\Delta u), D G_hv_h):=-\int_\Omega D^2(\Delta u) :DG_hv_h.
\end{eqnarray*}


Consequently,
\[
(\nabla(\Delta^2u), G_h v_h)=I_2- (D^2(\Delta u),  G_h^2 v_h),
\]
with
\begin{equation}\label{i2}
I_2=(D^2(\Delta u), G_h^2 v_h-DG_hv_h).
\end{equation}

Finally, we deal with the term $- (D^2(\Delta u),  G_h^2 v_h)$.
Writing the gradient as
\[
\nabla u=\frac{\partial u}{\partial {\bf n}}{\bf n}+\frac{\partial u}{\partial {\bf t}}{\bf t},
\]
we have
\[
D^2u=\frac{\partial^2 u}{\partial {\bf n}^2}{\bf n}{\bf n}^T +\frac{\partial^2 u}{\partial {\bf t}^2} {\bf t}{\bf t}^T+
\frac{\partial^2 u}{\partial {\bf n}\partial {\bf t}}({\bf n}{\bf t}^T+{\bf t}{\bf n}^T),
\]
and consequently,
\[
\frac{\partial D^2u}{\partial {\bf n}}=\frac{\partial^3 u}{\partial {\bf n}^3}{\bf n}{\bf n}^T +\frac{\partial^3 u}{\partial {\bf n}\partial {\bf t}^2} {\bf t}{\bf t}^T+
\frac{\partial^3 u}{\partial {\bf n}^2\partial {\bf t}}({\bf n}{\bf t}^T+{\bf t}{\bf n}^T).
\]
Noting that
$u=\frac{\partial u}{\partial {\bf n}}=\frac{\partial^2 u}{\partial {\bf n^2}}=0$ on $\partial\Omega$, we have
\[
\frac{\partial^3 u}{\partial {\bf n}\partial {\bf t}^2}=\frac{\partial^3 u}{\partial {\bf n}^2\partial {\bf t}}=0\ {\rm on}\ \Omega.
\]
Therefore, on $\partial\Omega$,
\[
\frac{\partial D^2u}{\partial {\bf n}} : G_h^2v_h=\frac{\partial^3 u}{\partial {\bf n}^3}{\bf n}{\bf n}^T:G_h^2v_h=\frac{\partial^3 u}{\partial {\bf n}^3}{\bf n}^TG_h^2v_h{\bf n}=0,
\]
where in the last equality, we used the fact that ${\bf n}^T G_h^2 v_h {\bf n}=0$.
By Green's formulas, we finally obtain
\[
- (D^2(\Delta u),  G_h^2 v_h)=- (\Delta (D^2 u),  G_h^2 v_h) =(D^3u, DG_h^2 v_h).
\]

In summary, by letting
\[
I_3=(D^3u-DG_h^2u_I, DG_h^2 v_h),
\]
we obtain
\begin{equation}
 a_h(u_h - u_I, v_h) =  I_1 + I_2 + I_3. \label{interpolationerror}
\end{equation}
Next, we estimate $I_i$ (i = 1, 2, 3) term by term.   By \eqref{a1} and \eqref{a2},  we have
\[
|I_1| \lesssim h \|u\|_6|G_hv_h|_1 \lesssim  h||u||_6||DG_h^2v_h||_0.
\]
Similarly, by \eqref{gr2},
\[
|I_3|  \lesssim h|u|_{4,\infty}||DG_h^2v_h||_0\lesssim h\|u\|_6||DG_h^2v_h||_0.
\]
On the other hand, \eqref{a3} implies
\[
|I_2| \lesssim   h\|u\|_6||DG_h^2v_h||_0.
\]

In a conclusion, we obtain that
\[
|a_h(u_h-u_I,v_h)|\lesssim h\|u\|_6\|DG_h^2v_h\|_0,\forall v_h\in V_h^0.
\]
Choosing $v_h=u_h-u_I$ in the above estimate, we have \eqref{h3}.

The estimate  \eqref{h30} is a direct consequence of \eqref{h3} and \eqref{gr2}, together with the triangle inequality.
\end{proof}

\subsection{$H^1$ error estimate}
In this section, we use the Aubin-Nitsche technique to estimate the $H^1$ norm error $\|\nabla u-G_hu_h\|_0$.
To this end, we construct the following auxiliary
problems :
\begin{itemize}
\item[1)] Find $U\in H_0^3(\Omega)$ such that
\begin{equation}\label{aux}
\int_{\Omega} D^3U:D^3 v=(G_h(u_h-u_I), \nabla v), \forall v\in H_0^3(\Omega).
\end{equation}
\item[2)]
Find $U_h\in V_h^{0}$ such that
\begin{equation}\label{auxdis}
a_h(U_h,v_h)=(G_h(u_h-u_I),\nabla v_h), \forall v_h\in V_h^{0}.
\end{equation}
\end{itemize}
 It is easy to deduce from
\eqref{aux}  that
\begin{equation}\label{auxbound}
\|U\|_5\lesssim \|G_h(u_h-u_I)\|_0,
\end{equation}
and from \eqref{auxdis},\eqref{a1} that
\begin{equation}\label{auxbounderr}
\|D(G_h^2 U_h)\|_0\lesssim \|G_h(u_h-u_I)\|_0.
\end{equation}

\begin{theorem}
\label{thm:h1error}
Let $u_h$ be the solution of \eqref{grs4ebc} and $u\in H^6$  the solution of \eqref{variational}.
 If the mesh $\T_h$ is  translation invariant, and $G_h$ is properly defined  such that \eqref{gr0}-\eqref{a3} hold, then
\begin{equation}\label{dh1}
\|G_h(u_h-u_I)\|_0\lesssim h^{\frac32}\|u\|_6.
\end{equation}
Consequently,
\begin{equation}\label{oh1}
\|\nabla u- G_hu_h\|_0\lesssim h^{\frac32}\|u\|_6.
\end{equation}
\end{theorem}
\begin{proof}First, by the definition of the auxiliary problems, we have
\begin{eqnarray*}
(G_h(u_h-u_I),\nabla (u_h-u_I))&=&a_h(U_h,u_h-u_I)\\&=&a_h(u_h-u_I,U_h)\\
&=&(f,U_h)-a_h(u_I,U_h)\\
&=&(-\triangle^3 u,U_h)-a_h(u_I,U_h).
\end{eqnarray*}
Using the same splitting techniques in the previous theorem, we can write
\[
(G_h(u_h-u_I),\nabla (u_h-u_I))=J_1+J_2+J_3+J_4,
\]
where
\begin{eqnarray*}
J_1&=& (\nabla(\triangle^2 u),\nabla U_h-G_h U_h),
\\
J_2&=& (D^2 (\triangle u),G_h^2 U_h-DG_h U_h),\\
J_3&=&(D^3 u-DG_h^2u_I, D^3U),\\
J_4&=&(D^3 u-DG_h^2u_I, DG_h^2U_h-D^3U).
\end{eqnarray*}
We first estimate $J_1$ and $J_3$. By \eqref{a3},
\begin{eqnarray*}
|J_1|&\lesssim&  h^2\|u\|_6\|DG_hU_h\|_0\\
&\lesssim& h^2\|u\|_6\|DG_h^2U_h\|_0\\
&\lesssim& h^2\|u\|_6\|G_h(u_h-u_I)\|_0.
\end{eqnarray*}
and
\begin{eqnarray*}
|J_2|
&\lesssim& h^2\|u\|_6\|DG_h^2U_h\|_0\\
&\lesssim& h^2\|u\|_6\|G_h(u_h-u_I)\|_0.
\end{eqnarray*}
Moreover, using the integration by parts,
\begin{eqnarray*}
|J_3|&\le &\|D^2u-G_h^2 u_I\|_0\|D^4U\|_0\\
&\lesssim& h^2\|u\|_6\|G_h(u_h-u_I)\|_0.
\end{eqnarray*}
Finally, by Cauchy-Schwartz inequality,
\begin{eqnarray*}
|J_4|&\lesssim& h^2\|u\|_6\|U\|_6\\
&\lesssim& h^2\|u\|_6\|G_h(u_h-u_I)\|_1\lesssim h^3\|u\|_6^2.
\end{eqnarray*}
Summarizing all the above estimates, we obtain
\[
(G_h(u_h-u_I),\nabla (u_h-u_I))\lesssim h^2\|u\|_6\|G_h(u_h-u_I)\|_0+h^3\|u\|_6^2.
\]
Noticing that
\[
\|G_h(u_h-u_I)\|_0^2\sim (G_h(u_h-u_I),\nabla (u_h-u_I)),
\]
we arrive that
\[
\|G_h(u_h-u_I)\|_0^2\lesssim h^2\|u\|_6 \|G_h(u_h-u_I)\|_0+h^3\|u\|_6^2.
\]
Then the estimate \eqref{dh1} follows.

The $H^1$ error estimate of \eqref{oh1} is a direct consequence of \eqref{dh1} and \eqref{gr1}.
\end{proof}
\subsection{$L^2$ error estimate}

\begin{theorem}
\label{thm:l2error}
Let $u_h$ be the solution of \eqref{grs4ebc} and $u\in H^6$  the solution of \eqref{variational}.
 If the mesh $\T_h$ is sufficiently regular and uniform, $G_h$ is properly defined  such that \eqref{gr0}-\eqref{a3} hold, then
\begin{gather}
\|u-u_h\|_0\lesssim h^{\frac32}\|u\|_6.\label{l2}\end{gather}
\end{theorem}
\begin{proof} By \eqref{a1},
\begin{eqnarray*}
\|u_I-u_h\|_0&\lesssim &\|G_h(u_I-u_h)\|_0\lesssim h^{\frac32}\|u\|_6.
\end{eqnarray*}
Then
\begin{eqnarray*}
\|u-u_h\|_0&\lesssim&\|u-u_I\|_0+\|u_I-u_h\|_0 \lesssim h^{\frac32}\|u\|_6.
\end{eqnarray*}
\end{proof}

\begin{remark} In the second section, we observed the convergence rates ${\cal O}(h^2)$ both for the errors
$\|u-u_h\|_{0}$ and $\|\nabla u-G_hu_h\|_{0}$. However, we can only prove the order ${\cal O}(h^{3\over 2})$ from our analysis.
Further analysis to the scheme is desired to prove the optimal convergence rates of $\|u-u_h\|_{0}$ and $\|\nabla u-G_hu_h\|_{0}$.
\end{remark}

\section{Numerical Experiments}
In this section, we present several numerical experiments to show the convergence rates and
efficiency of our method. In all our numerical experiments,  $G_h$ is chosen as the polynomial preserving recovery operator \cite{ZhangNaga2005}.
To present our numerical results, the following notations are used :
\begin{eqnarray*}
De&:=&\|u-u_h\|_{0},\quad\quad\quad D^1e:=\|\nabla u- \nabla u_h\|_{0},\\
D^1_re&:=&\|\nabla u-G_hu_h\|_{0},\quad
D^2e:=\|D^2 u-DG_hu_h\|_{ 0}, \\
D^3e&:=&\|D^3u - DG_h^2u_h\|_{ 0}.
 \end{eqnarray*}
Moreover, the convergence rates are listed  with respect to the degree of freedom(Dof).
Noticing  $\text{Dof}  \approx h^{-2}$ for a two dimensional grid,  the
corresponding  convergent rates
with respect to the mesh size $h$ are double of what we present in the tables 3.1-3.11.

{\bf Example 1}.  We consider the triharmonic problem
\begin{equation}
\left\{
\begin{array}{ll}
  -\Delta^3u = f&  \text{in } \Omega = (0, 1)\times(0, 1);    \\
  u=\partial_{\bf n} u=\partial^2_{\bf nn} u=0 & \text{on } \partial \Omega,
\end{array}
\right.
\end{equation}
where $f$ is chosen to fit the exact solution $u(x_1, x_2) = x_1^3(1-x_1)^3x_2^3(1-x_2)^3$.

First, we apply our scheme \eqref{grs4ebc} on regular pattern uniform  triangular mesh.
The corresponding numerical results are listed in Table
\ref{tab:neilanregular}.
It shows that  numerical solution $u_h$ converges to the exact solution $u$
at a rate of $O(h)$ in
recovered $H^3$ norm. 
Also from this table, we observe that $De$ and $D^1_re$ converge
at a rate of $O(h^2)$ while $D^1e$ and $D^2e$ converge at  a rate of $O(h)$.
Note that convergence rate
of $De$ and $D^1_re$ is better than that proved in Theorems \ref{thm:h1error} and \ref{thm:l2error}.

\begin{table}[htb!]
\centering
\footnotesize
\caption{Numerical Results Of Example 1 On Regular Pattern Uniform Mesh}\begin{tabular}{|c|c|c|c|c|c|c|c|c|c|c|}
\hline
 Dof & $De$ & order& $D^{1}e$ & order& $D^{1}_re$ & order& $D^2e$ &  order& $ D^3e$ & order\\ \hline\hline
 1089 &5.61e-06&--&5.76e-05&--&2.57e-05&--&4.01e-04&--&4.46e-03&--\\ \hline
 4225 &1.54e-06&0.95&2.20e-05&0.71&7.03e-06&0.96&1.73e-04&0.62 &2.06e-03&0.57 \\ \hline
 16641 &3.99e-07&0.99&9.65e-06&0.60&1.83e-06&0.98&8.18e-05&0.54 &9.99e-04&0.53 \\ \hline
 66049 &1.01e-07&0.99&4.62e-06&0.53&4.66e-07&0.99&4.03e-05&0.51 &4.93e-04&0.51 \\ \hline
\end{tabular}
\label{tab:neilanregular}
\end{table}

Secondly, we test our scheme on uniform triangular meshes of other patterns, including the chevron, Criss-cross, and Union-Jack patterns.
Numerical data are listed in \ref{tab:neilanchevron},
Table \ref{tab:neilancrisscross}, and Table \ref{tab:neilanunionjack}, respectively.
 Again, we observed $O(h^2)$ for $De$, $O(h)$ for $D^1e$, $O(h^2)$ for $D^1_re$,  $O(h)$ for
 $D^2e$, and $O(h)$ for $D^3e$,   the same as the regular pattern. 

\begin{table}[htb!]
\centering
\footnotesize
\caption{Numerical Results Of Example 1 On Chevron Pattern Uniform Mesh}\begin{tabular}{|c|c|c|c|c|c|c|c|c|c|c|}
\hline
 Dof & $De$ & order& $D^{1}e$ & order& $D^{1}_re$ & order& $D^2e$ &  order& $ D^3e$ & order\\ \hline\hline
 1089 &4.48e-06&--&6.00e-05&--&2.02e-05&--&3.81e-04&-- &4.26e-03&-- \\ \hline
 4225 &1.25e-06&0.94&2.22e-05&0.73&5.65e-06&0.94&1.69e-04&0.60 &2.04e-03&0.54 \\ \hline
 16641 &3.24e-07&0.99&9.67e-06&0.61&1.47e-06&0.98&8.14e-05&0.53 &9.97e-04&0.52 \\ \hline
 66049 &8.24e-08&0.99&4.62e-06&0.54&3.75e-07&0.99&4.03e-05&0.51 &4.92e-04&0.51 \\ \hline
\end{tabular}
\label{tab:neilanchevron}
\end{table}

\begin{table}[htb!]
\centering
\footnotesize
\caption{Numerical Results Of Example 1 On Criss-cross Pattern Uniform Mesh}\begin{tabular}{|c|c|c|c|c|c|c|c|c|c|c|}
\hline
 Dof & $De$ & order& $D^{1}e$ & order& $D^{1}_re$ & order& $D^2e$ &  order& $ D^3e$ & order\\ \hline\hline
 2113 &1.50e-05&--&2.91e-03&--&3.13e-05&--&3.61e-03&-- &7.21e-03&-- \\ \hline
 8321 &3.83e-06&1.00&1.48e-03&0.50&8.02e-06&0.99&1.83e-03&0.50 &3.62e-03&0.50 \\ \hline
 33025 &9.46e-07&1.01&7.20e-04&0.52&2.02e-06&1.00&8.98e-04&0.52 &1.81e-03&0.50 \\ \hline
 131585 &2.39e-07&0.99&3.64e-04&0.49&5.09e-07&1.00&4.55e-04&0.49 &9.07e-04&0.50 \\ \hline
\end{tabular}
\label{tab:neilancrisscross}
\end{table}

\begin{table}[htb!]
\centering
\footnotesize
\caption{Numerical Results Of Example 1 On Unionjack Pattern Uniform Mesh}\begin{tabular}{|c|c|c|c|c|c|c|c|c|c|c|}
\hline
 Dof & $De$ & order& $D^{1}e$ & order& $D^{1}_re$ & order& $D^2e$ &  order& $ D^3e$ & order\\ \hline\hline
 1089 &2.61e-05&--&3.24e-03&--&6.30e-05&--&4.33e-03&-- &9.40e-03&-- \\ \hline
 4225 &6.53e-06&1.02&1.61e-03&0.51&1.59e-05&1.02&2.12e-03&0.53 &4.52e-03&0.54 \\ \hline
 16641 &1.63e-06&1.01&8.03e-04&0.51&4.00e-06&1.01&1.06e-03&0.51 &2.22e-03&0.52 \\ \hline
 66049 &4.08e-07&1.01&4.01e-04&0.50&1.00e-06&1.00&5.26e-04&0.50 &1.10e-03&0.51 \\ \hline
\end{tabular}
\label{tab:neilanunionjack}
\end{table}

Finally, we turn to  the Delaunay mesh.
The first level coarse mesh is generated by EasyMesh \cite{easymesh}
 followed by three levels of regular refinement.
Table \ref{tab:neilandelaunay} presents the convergence history for the five different errors.
$O(h^2)$ and $O(h)$ convergence rates
are observed for  $L_2$ and $H_1$ errors.    As for  the $L_2$ error of recovered
gradient,  $O(h^2)$ superconvergence is observed.
Regarding recovered $H_2$ and $H_3$ errors,  $O(h)$ convergence are observed .

 In summary, we see that our method converges with optimal rates on all four tested uniform meshes as well as the Delaunay mesh.

 To show the efficiency of  our method, we make some numerical comparison
with  the cubic $C^0$ interior penalty method \cite{GudiNeilan2011} on the same Delaunay meshes.
Table \ref{tab:c0ip} shows numerical results of  the $C^0$ interior penalty method in the $L_2$ norm and the energy norm
defined in \cite{GudiNeilan2011}.
 Consisting with the theoretical result established  in \cite{GudiNeilan2011}, the  error in the energy norm  converges linearly
 and the $L_2$ error decays  at rate $O(h^2)$.

Figures \ref{fig:compareh3} and  \ref{fig:comparel2}  depict convergent rates of these two methods(i.e. our method and the $C^0$ interior penalty method)  under the discrete $H_3$ (the energy) and $L_2$ norms.
The rates are almost the same.  However, to achieve the same accuracy, our algorithm uses about one-eighth degrees of freedom of the  $C^0$ interior penalty method.

\begin{table}[htb!]
\centering
\footnotesize
\caption{Numerical Results of Example 1 on  Delaunay Triangulation with Regular Refinement}\begin{tabular}{|c|c|c|c|c|c|c|c|c|c|c|}
\hline
 Dof & $De$ & order& $D^{1}e$ & order& $D^{1}_re$ & order& $D^2e$ &  order& $ D^3e$ & order\\ \hline\hline
 513 &1.08e-05&--&9.72e-05&--&4.93e-05&--&6.21e-04&--&6.34e-03&-- \\ \hline
 1969 &3.02e-06&0.95&3.04e-05&0.86&1.38e-05&0.95&2.41e-04&0.70 &2.90e-03&0.58 \\ \hline
 7713 &7.94e-07&0.98&1.23e-05&0.66&3.65e-06&0.97&1.08e-04&0.59 &1.37e-03&0.55 \\ \hline
 30529 &2.03e-07&0.99&5.77e-06&0.55&9.35e-07&0.99&5.20e-05&0.53 &6.67e-04&0.52 \\ \hline
\end{tabular}
\label{tab:neilandelaunay}
\end{table}

\begin{table}[htb!]
\centering
\caption{$C^0$ Interior Penalty Method for Example 1 on  Delaunay Triangulation with Regular Refinement
}
\begin{tabular}{|c|c|c|c|c|c|c|c|c|c|c|}
\hline
 Dof & $De$ & order&  $ D^3_he$ & order\\ \hline\hline
 4369 &7.75e-06&--&6.25e-03&--\\ \hline
 17233 &2.72e-06&0.76 &2.92e-03&0.55 \\ \hline
 68449 &7.89e-07&0.90 &1.37e-03&0.55 \\ \hline
 272833 &1.58e-07&1.16&6.63e-04&0.52 \\ \hline
\end{tabular}
\label{tab:c0ip}
\end{table}

\begin{figure}[!h]
  \centering
  \begin{minipage}[c]{0.5\textwidth}
  \centering
    \includegraphics[width=\textwidth]{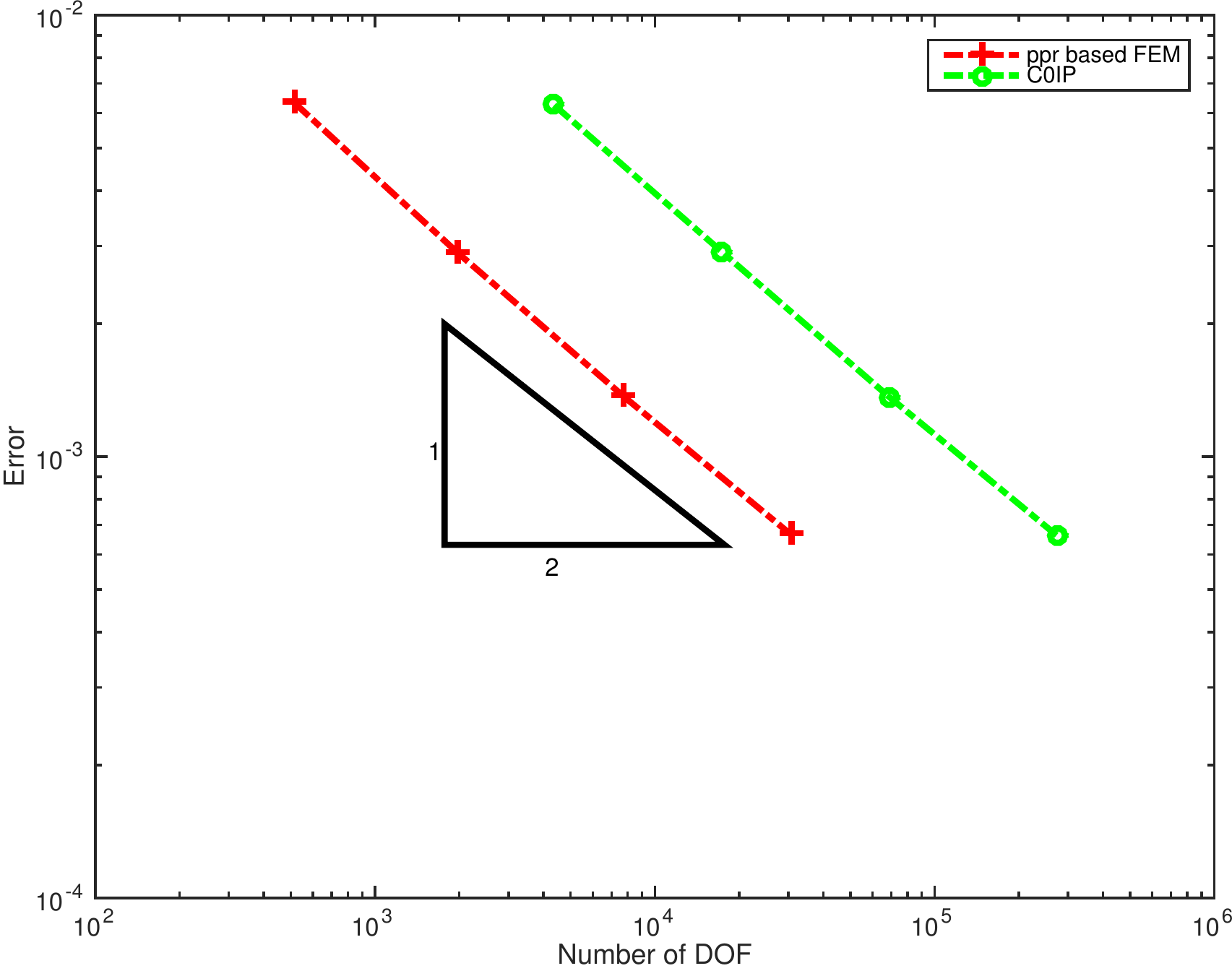}
     \caption{Comparison  of Discrete $H_3$ Errors for Example 1}
\label{fig:compareh3}
\end{minipage}%
 \begin{minipage}[c]{0.5\textwidth}
  \centering
    \includegraphics[width=\textwidth]{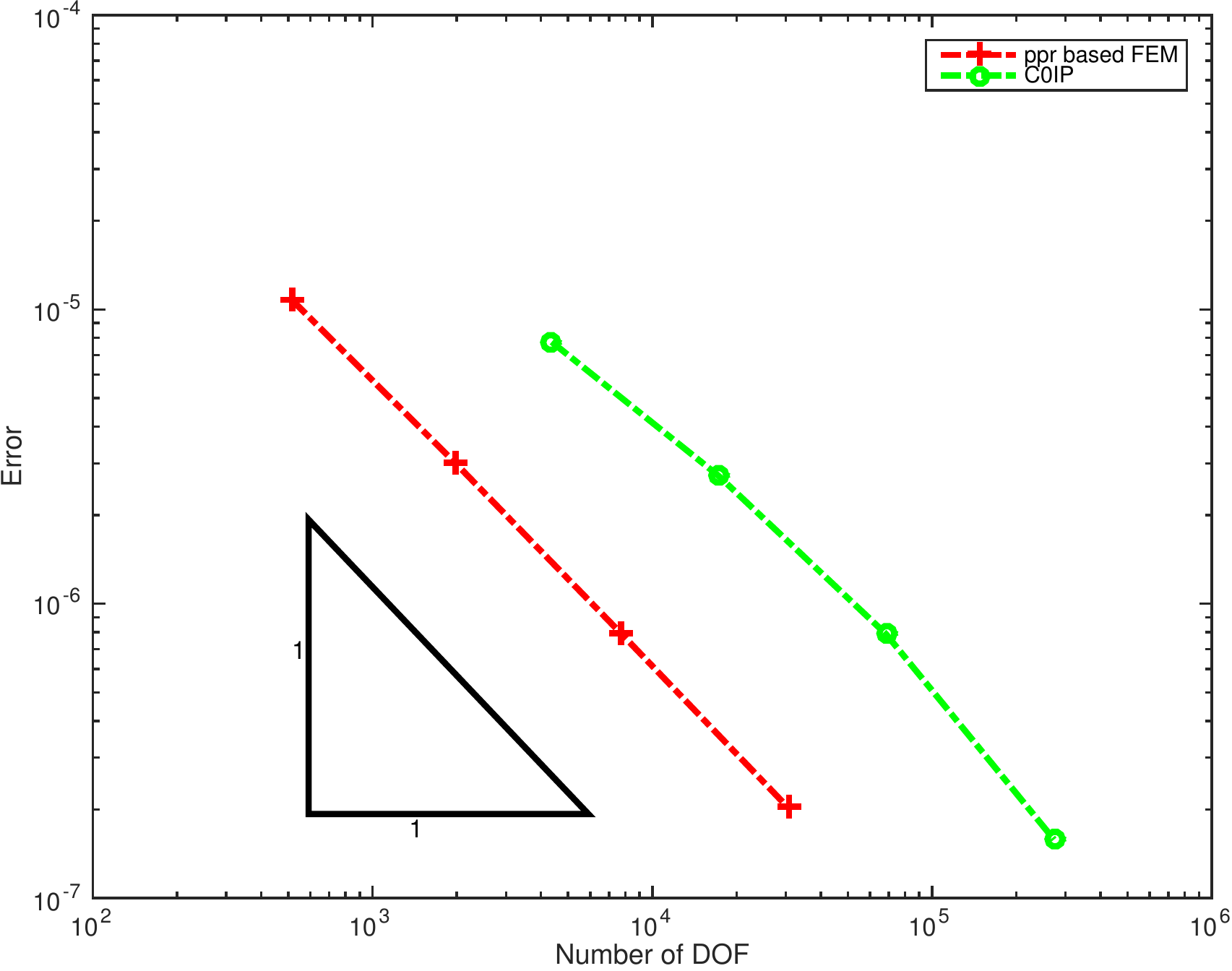}
   \caption{Comparison of  Discrete $L_2$ Errors for Example 1}
\label{fig:comparel2}
\end{minipage}
\end{figure}

{\bf Example 2.} In the second example, we will
show that our scheme  works well also
for problems nonhomogeneous boundary conditions.
We consider the  equation
\begin{equation*}
-\Delta^3 u = \sin(2\pi x_1)\cos(2\pi x_2),(x_1,x_2)\in[0,1]^2
\end{equation*}
whose exact solution is
\begin{equation*}
u(x_1, x_2) = \frac{1}{512\pi^6}\sin(2\pi x_1)\cos(2\pi x_2),
\end{equation*}
It provides nonhomogeneous boundary conditions
$u|_{\partial\Omega}, \partial_{\bf n}u|_{\partial\Omega}, \partial^2_{\bf nn} u|_{\partial\Omega}$.

As in Example 1, we first test our algorithm on regular pattern uniform triangular mesh and list
the numerical results in Table \ref{tab:sinregular}.   Again,
$D^3e$ decays at rate $O(h)$.
As expected, both $De$ and $D^1_re$ converges with order $O(h^2)$. Astonishingly, both $D^1e$ and $D^2e$ also converge quadratically. Namely,
for this example, both $D^1e$ and $D^2e$ superconverge.
\begin{table}[htb!]
\centering
\footnotesize
\caption{Numerical Results Of Example 2 On Regular Pattern Uniform Mesh}\begin{tabular}{|c|c|c|c|c|c|c|c|c|c|c|}
\hline
 Dof & $De$ & order& $D^{1}e$ & order& $D^{1}_re$ & order& $D^2e$ &  order& $ D^3e$ & order\\ \hline\hline
 1089 &2.50e-07&--&3.21e-05&--&1.53e-06&--&1.35e-04&-- &2.25e-03&-- \\ \hline
 4225 &2.66e-08&1.65&7.09e-06&1.11&1.47e-07&1.73&2.93e-05&1.13 &7.65e-04&0.80 \\ \hline
 16641 &3.01e-09&1.59&1.46e-06&1.15&1.92e-08&1.48&6.75e-06&1.07 &3.24e-04&0.63 \\ \hline
 66049 &4.75e-10&1.34&3.32e-07&1.08&3.41e-09&1.25&1.76e-06&0.98 &1.37e-04&0.62 \\ \hline
\end{tabular}
\label{tab:sinregular}
\end{table}

We then  consider chevron pattern uniform  triangular mesh.
Table \ref{tab:sinchevron} clearly indicates that $u_h$ converges to $u$
at a rate of $O(h^2)$ under the $L^2$ norm,
at a rate of $O(h)$ under the $H^1$ norm
and the recovered $H^2$ and $H^3$  norms.
Moreover, the recovery gradient $G_hu_h$ converges to $\nabla u$ at a rate
of $O(h^2)$.
We also  test our algorithms on   Delaunay meshes as
in the previous example. The numerical  data are demonstrated in Table \ref{tab:sindelaunay}.
Similar to what  we observed in Chevron pattern uniform  triangular mesh, the computed error by our method
converges to 0
with optimal rates under various norms.

In addition,  we have tested  our algorithms  on other two types (Criss-cross and Union-Jack pattern)
uniform triangular meshes.
Since the numerical results are similar to  the corresponding parts in the previous example, they are not reported here.
\begin{table}[htb!]
\centering
\footnotesize
\caption{Numerical Results of Example 2 on Chevron Pattern Uniform Mesh}\begin{tabular}{|c|c|c|c|c|c|c|c|c|c|c|}
\hline
 Dof & $De$ & order& $D^{1}e$ & order& $D^{1}_re$ & order& $D^2e$ &  order& $ D^3e$ & order\\ \hline\hline
 1089 &4.39e-08&--&1.08e-06&--&4.20e-07&--&1.23e-05&--&2.70e-04&-- \\ \hline
 4225 &3.38e-09&1.89&4.47e-07&0.65&3.92e-08&1.75&3.65e-06&0.90 &9.95e-05&0.74 \\ \hline
 16641 &7.89e-10&1.06&2.22e-07&0.51&7.57e-09&1.20&1.61e-06&0.60 &4.13e-05&0.64 \\ \hline
 66049 &2.00e-10&1.00&1.11e-07&0.50&1.83e-09&1.03&7.84e-07&0.52 &1.78e-05&0.61 \\ \hline
\end{tabular}
\label{tab:sinchevron}
\end{table}

\begin{table}[htb!]
\centering
\footnotesize
\caption{Numerical Results of Example 2 on Delaunay Triangulation with Regular Refinement}\begin{tabular}{|c|c|c|c|c|c|c|c|c|c|c|}
\hline
 Dof & $De$ & order& $D^{1}e$ & order& $D^{1}_re$ & order& $D^2e$ &  order& $ D^3e$ & order\\ \hline\hline
 513 &5.52e-08&--&1.77e-06&--&3.81e-07&--&1.29e-05&-- &1.44e-04&-- \\ \hline
 1969 &1.17e-08&1.15&5.43e-07&0.88&8.07e-08&1.15&4.21e-06&0.83 &5.78e-05&0.68 \\ \hline
 7713 &2.79e-09&1.05&2.57e-07&0.55&1.97e-08&1.03&1.90e-06&0.58 &1.94e-05&0.80 \\ \hline
 30529 &6.86e-10&1.02&1.27e-07&0.51&4.91e-09&1.01&9.40e-07&0.51 &8.52e-06&0.60 \\ \hline
\end{tabular}
\label{tab:sindelaunay}
\end{table}

Once again, we present a numerical comparison with the $C^0$ interior penalty method.
We see from Figure \ref{fig:h3comp} that the convergence rates of $H_3$ error are comparable, however, our method requires much less
degrees of freedom in order to achieve the same accuracy.
Figure \ref{fig:l2comp} indicates that our method is  slightly better than the $C^0$ interior penalty method
with regard to  the $L_2$ norm which is suboptimal.  Here we would like to point out that the error of
the $C^0$ interior penalty method is  sensitive to  the penalty parameter.

\begin{figure}[!h]
  \centering
  \begin{minipage}[c]{0.5\textwidth}
  \centering
    \includegraphics[width=\textwidth]{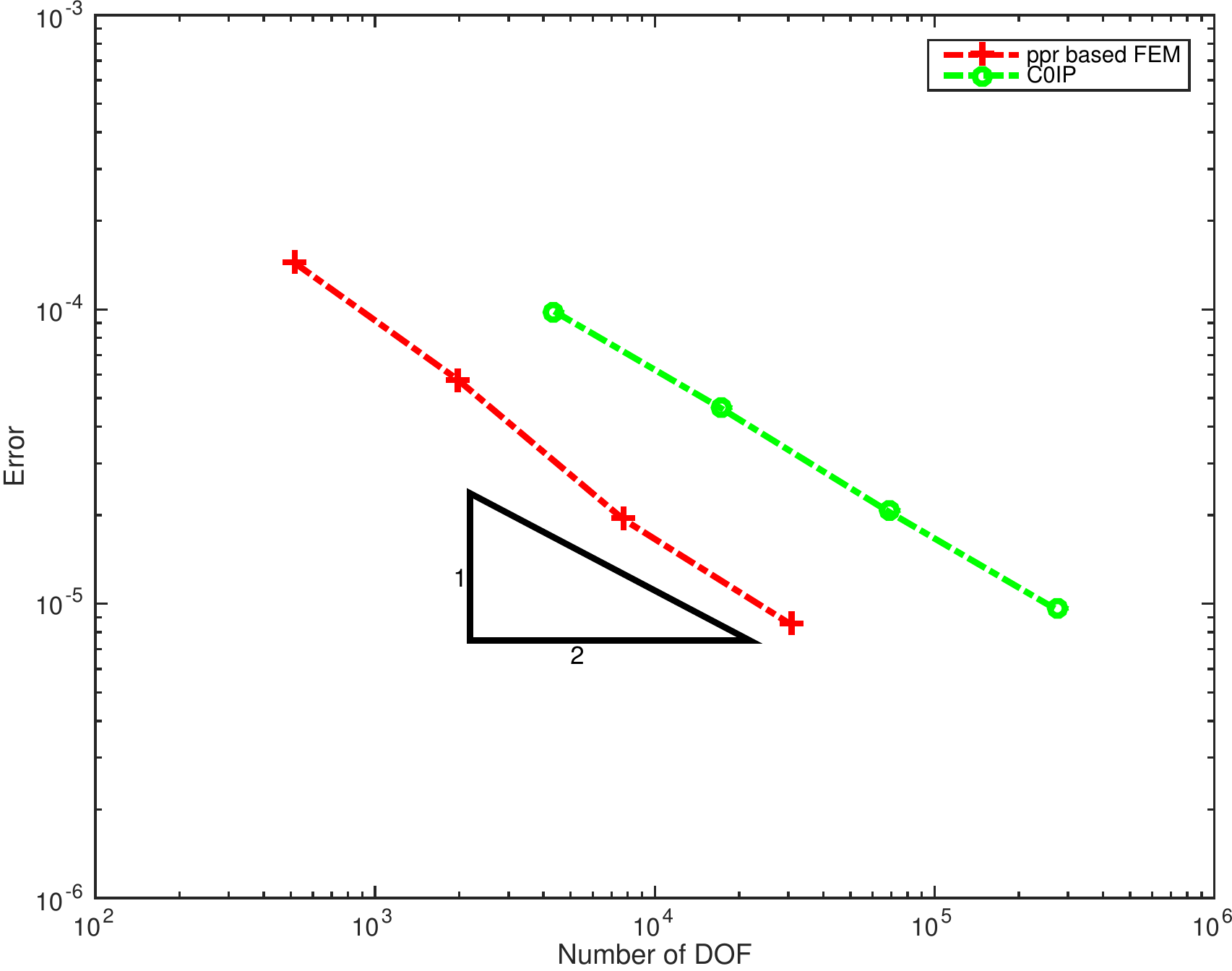}
     \caption{Comparison  of Discrete $H_3$ Errors for Example 2}
\label{fig:h3comp}
\end{minipage}%
 \begin{minipage}[c]{0.5\textwidth}
  \centering
    \includegraphics[width=\textwidth]{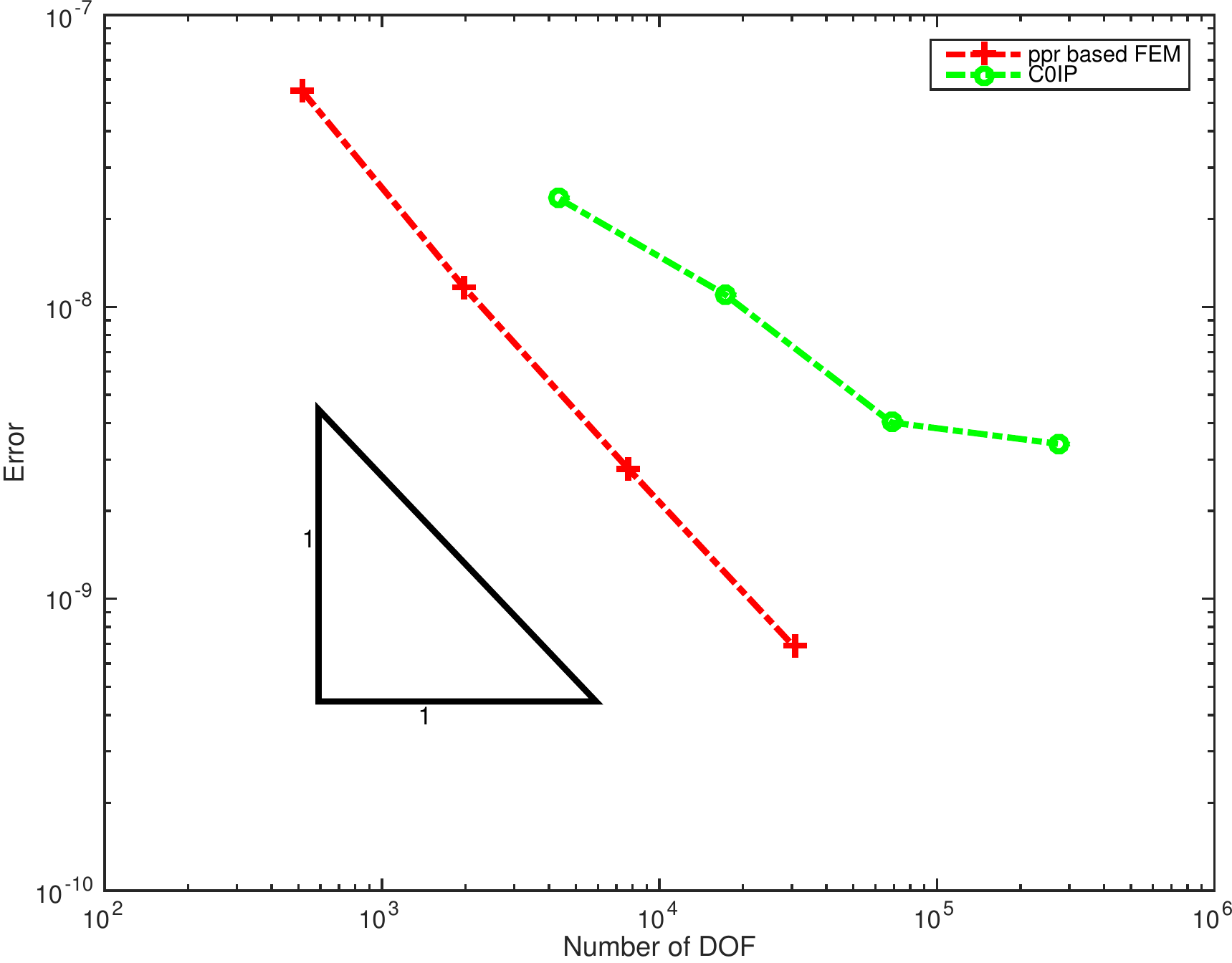}
   \caption{Comparison of  Discrete $L_2$ Errors for Example 2}
\label{fig:l2comp}
\end{minipage}
\end{figure}

{\bf Example 3.} In previous two examples, we consider sixth order elliptic  equations  on the unit square. To show the ability  of dealing arbitrary
complex domain, we consider the following sixth order partial differential equation
\begin{equation*}
-\Delta^3 u =  8 e^{x_1+x_2}.
\end{equation*}
on the unit disk, i.e. $\Omega = \{(x_1,x_2) \in \mathbb{R}^2: x_1^2+x_2^2 \le 1\}$.
The exact solution is
\begin{equation*}
u(x_1, x_2) = e^{x_1+x_2}.
\end{equation*}
and the corresponding boundary conditions are given by the exact solution.
The initial mesh is generated by  DistMesh \cite{PerssonStrang2004} as shown
in  Figure \ref{fig:circle}.
The other seven level meshes are obtained by refining the initial mesh using regular refinement.
The numerical results are reported in Table \ref{tab:circle}.
As in two previous examples,
$O(h)$ convergence for $D^3e$ ,  $D^2e$, and $D^1e$ are observed and
 $O(h^2)$  convergence order are observed for
$De$ and $D^1_re$.

\begin{figure}[!h]
  \centering
  \begin{minipage}[c]{0.5\textwidth}
  \centering
    \includegraphics[width=0.9\textwidth]{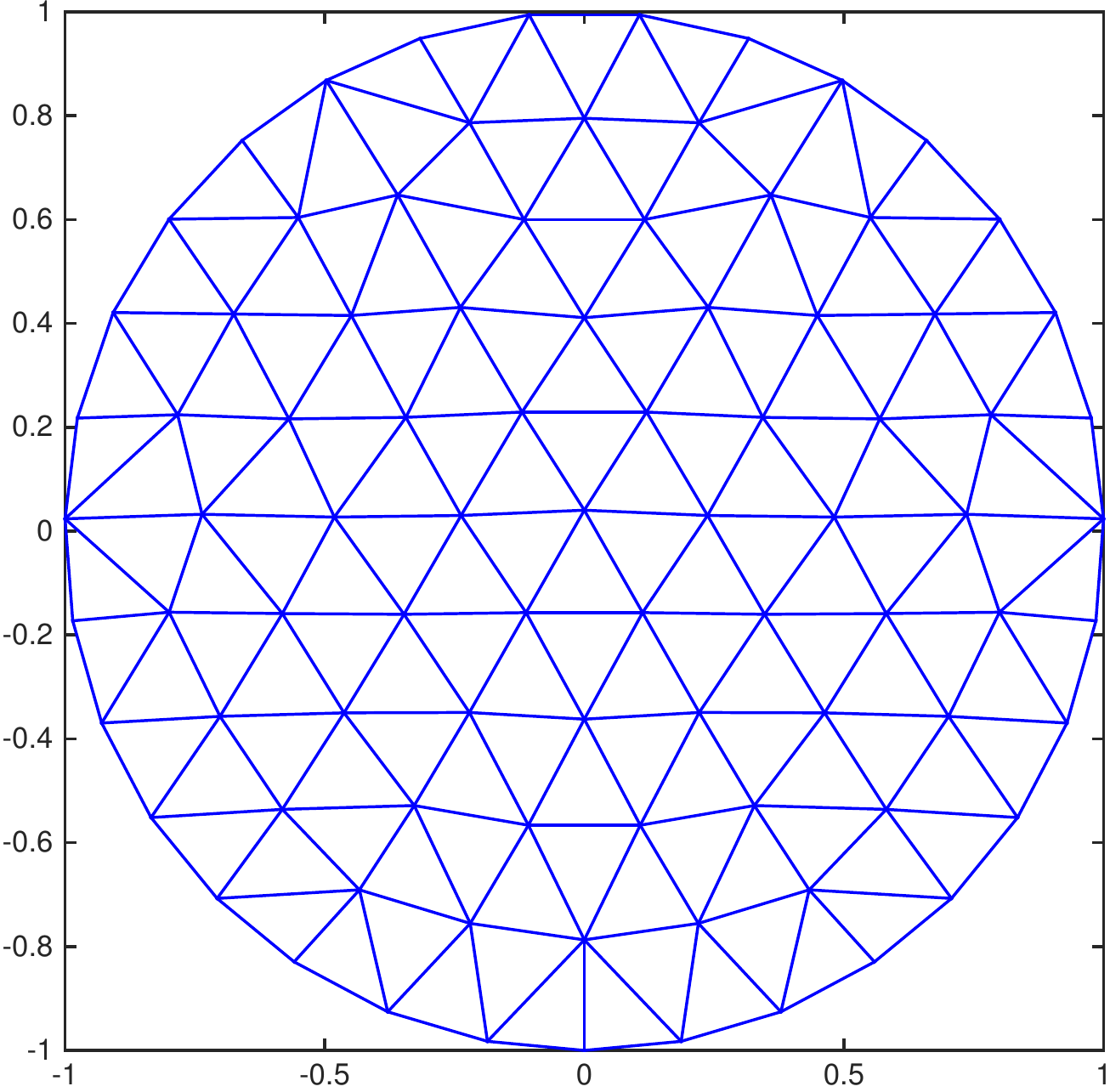}
   \caption{Initial Mesh On The Unit Disk}
\label{fig:circle}
\end{minipage}%
 \begin{minipage}[c]{0.5\textwidth}
  \centering
    \includegraphics[width=0.9\textwidth]{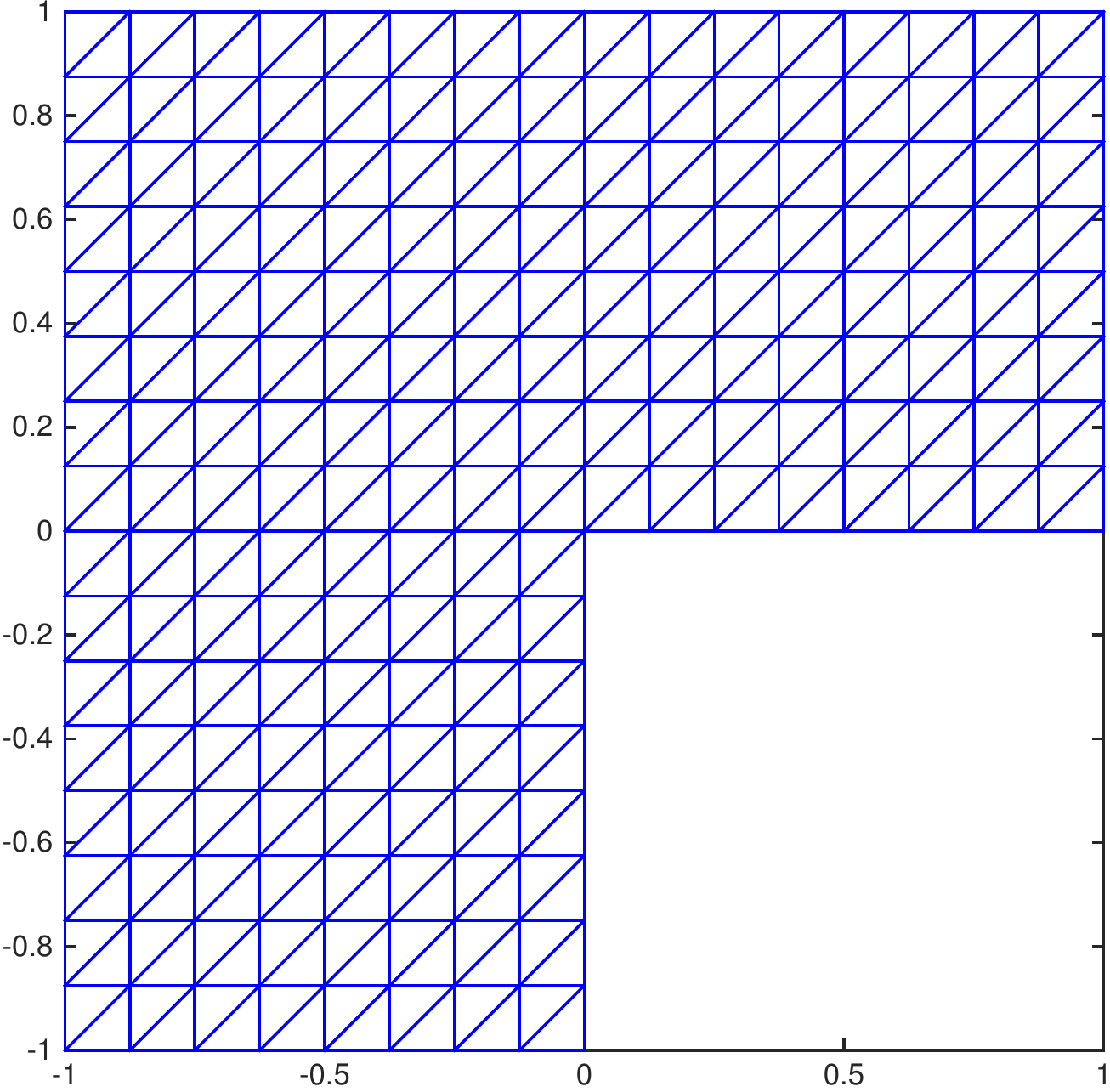}
   \caption{Initial Mesh On The Lshape Domain}
\label{fig:lshape}
\end{minipage}
\end{figure}

\begin{table}[htb!]
\centering
\footnotesize
\caption{Numerical Results  of  Sixth Order PDE  in the Unit Disk}
 \begin{tabular}{|c|c|c|c|c|c|c|c|c|c|c|}
\hline
 Dof & $De$ & order& $D^{1}e$ & order& $D^{1}_re$ & order& $D^2e$ &  order& $ D^3e$ & order\\ \hline\hline
 88 &3.69e-02&--&4.76e-01&--&1.38e-01&--&1.01e+00&-- &3.70e+00&-- \\ \hline
 318 &1.44e-02&0.73&1.91e-01&0.71&3.21e-02&1.14&3.48e-01&0.83 &1.92e+00&0.51 \\ \hline
 1207 &1.84e-03&1.54&8.05e-02&0.65&4.76e-03&1.43&1.06e-01&0.89 &7.14e-01&0.74 \\ \hline
 4701 &4.16e-04&1.09&4.00e-02&0.51&1.20e-03&1.01&4.94e-02&0.56 &3.66e-01&0.49 \\ \hline
 18553 &1.04e-04&1.01&2.00e-02&0.51&3.03e-04&1.00&2.40e-02&0.52 &1.99e-01&0.44 \\ \hline
 73713 &2.60e-05&1.00&1.00e-02&0.50&7.57e-05&1.01&1.19e-02&0.51 &9.70e-02&0.52 \\ \hline
 293857 &7.01e-06&0.95&5.00e-03&0.50&1.85e-05&1.02&5.94e-03&0.50 &4.62e-02&0.54 \\ \hline
\end{tabular}
\label{tab:circle}
\end{table}

{\bf Example 4.} As in  \cite{HuZhang2015}, we consider the following triharmonic equation
\begin{equation*}
-\Delta^3 u =  0.
\end{equation*}
on the L-shaped domain $[-1,1]^2\setminus([0,1]\times[-1,0])$ with boundary conditions such that the problem
 has the exact solution
\begin{equation*}
u(x_1, x_2) = x_1^6-x_2^6.
\end{equation*}
Here we use uniform meshes.  The initial mesh is plotted in Figure \ref{fig:lshape}, while
our numerical results are listed in Table \ref{tab:lshape}.
As pointed out in \cite{HuZhang2015}, the solution $u$ varies fast near the boundary.
Even in that case,
we  observe the optimal convergence rates under  all the  norms.

\begin{table}[htb!]
\centering
\footnotesize
\caption{Numerical Results For Triharmonic Equation On LShape Domain}
 \begin{tabular}{|c|c|c|c|c|c|c|c|c|c|c|}
\hline
 Dof & $De$ & order& $D^{1}e$ & order& $D^{1}_re$ & order& $D^2e$ &  order& $ D^3e$ & order\\ \hline\hline
 225 &2.04e-01&--&5.78e+00&--&8.10e-01&--&1.42e+01&-- &7.58e+01&-- \\ \hline
 833 &2.63e-02&1.56&1.69e+00&0.94&1.24e-01&1.43&4.12e+00&0.95 &2.96e+01&0.72 \\ \hline
 3201 &3.54e-03&1.49&4.20e-01&1.03&2.70e-02&1.14&1.42e+00&0.79 &1.35e+01&0.58 \\ \hline
 12545 &6.66e-04&1.22&1.39e-01&0.81&6.67e-03&1.02&5.71e-01&0.66 &6.70e+00&0.52 \\ \hline
 49665 &1.44e-04&1.12&5.95e-02&0.62&1.71e-03&0.99&2.62e-01&0.57 &3.35e+00&0.50 \\ \hline
 197633 &3.35e-05&1.05&2.82e-02&0.54&4.55e-04&0.96&1.27e-01&0.52 &1.68e+00&0.50 \\ \hline
\end{tabular}
\label{tab:lshape}
\end{table}

In summary, our numerical experiments discover that  our algorithm
converges with optimal rates under various norms, for sixth order
equations on different kinds of domains, with homogenous or
nonhomogeneous boundary conditions. In addition, comparing to
some existed algorithm such as $C^0$ interior penalty method,
our algorithm has much lower computational cost.

\section{Concluding remarks}
 In this work, we developed a PPR based discretization algorithm for a  sixth-order PDE.
The algorithm has a simple form and is easy to implement. Moreover, it has optimal convergence rates
as the existing conforming and nonconforming FEMs in the literatures for sixth-order PDEs.
However, the  new method seems to be more advantageous with respect to computational complexity.

Generally speaking, the recovery operator is  a special difference operator on nonuniform grids.
It can be used to {\it compute} high order derivatives of a function which are piecewise polynomials
but only globally in $C^0$ and thus can be used to discretize  PDEs of higher order.
On the other hand, how to choose different recovery operators for different PDEs  deserves more in-depth
mathematical study. Further investigation is called for to find simple and efficient
 algorithms for complicated PDEs.

\end{document}